\documentclass[a4paper,twoside,11pt]{article}
\usepackage{a4wide,graphicx,fancyhdr,amsmath,amssymb,mathtools,yfonts}
\usepackage[all]{xy}
\usepackage[utf8]{inputenc}
\usepackage{amsthm}
\usepackage[english]{babel}
\usepackage{chngcntr}
\usepackage{ifthen}
\usepackage{calc}
\usepackage{hyperref}

\counterwithin*{equation}{section}
\counterwithin*{equation}{subsection}

\newcommand{\N}{\mathbb{N}}
\newcommand{\Z}{\mathbb{Z}}

\fancypagestyle{plain}{%
\fancyhf{}
\fancyhead[LO,RE]{\sffamily\bfseries\large Leiden University}
\fancyhead[RO,LE]{\sffamily\bfseries\large Research}
\fancyfoot[LO,RE]{\sffamily\bfseries\large Mathematical Institute}
\fancyfoot[RO,LE]{\sffamily\bfseries\thepage}

}

\newtheorem{theorem}{Theorem}
\newtheorem{lemma}[theorem]{Lemma}

\newcounter{Constant}[chapter]
\newcounter{constant}[chapter]

\newcommand{\Cst}[1]
{
\ifthenelse{\value{#1}=0}
  {\addtocounter{Constant}{1}\setcounter{#1}{\value{Constant}}f_{\arabic{#1}}}
  {f_{\arabic{#1}}}
}

\newcommand{\cst}[1]
{
\ifthenelse{\value{#1}=0}
  {\addtocounter{constant}{1}\setcounter{#1}{\value{constant}}c_{\arabic{#1}}}
  {c_{\arabic{#1}}}
}

\title{\vspace{-\baselineskip}\sffamily\bfseries The generalized Catalan equation in positive characteristic}
\author{Peter Koymans \\{\tt p.h.koymans@math.leidenuniv.nl}}

\date{\today}

\begin{document}
\maketitle
\label{poschar}
\noindent In this article we will bound $m$ and $n$ for the generalized Catalan equation in characteristic $p > 0$. \\

\noindent \textbf{Notation} \\
Let $K = \mathbb{F}_p(z_1, \ldots, z_r)$ be a finitely generated field over $\mathbb{F}_p$. Our main result in this article is as follows.

\begin{theorem}
Let $a, b \in K^\ast$ be given. Consider the equation
\begin{equation}
\label{Cat}
ax^m + by^n = 1
\end{equation}
in $x, y \in K$ and integers $m, n > 1$ coprime with $p$ satisfying
\begin{align}
\label{mn}
(m, n) \not \in \{(2, 2), (2, 3), (3, 2), (2, 4), (4, 2), (3, 3)\}.
\end{align}
Then there is a finite set $\mathcal{T} \subseteq K^2$ such that for any solution $(x, y, m, n)$ of (\ref{Cat}), there is a $(\gamma, \delta) \in \mathcal{T}$ and $t \in \mathbb{Z}_{\geq 0}$ such that
\begin{align}
\label{result}
ax^m = \gamma^{p^t}, by^n = \delta^{p^{t}}.
\end{align}
\end{theorem}

\noindent \textbf{Discussion of Theorem 1} \\
In the case $a = b = 1$, a stronger and effective result was proven in \cite{MThesis} based on the work of \cite{Brindza}.

Let us now show that the conditions on $m$ and $n$ are necessary. If (\ref{mn}) fails, then (\ref{Cat}) defines a curve of genus 0 or 1 over $K$. It is clear that (\ref{result}) can fail in this case. It is also essential that $m$ and $n$ are coprime with $p$. Take for example $a = b = 1$. Then any solution of
\[
x + y = 1
\]
with $x, y \in K$ and $x, y \not \in \overline{\mathbb{F}_p}$ gives infinitely many solutions of the form (\ref{result}) after applying Frobenius.

The generalized Catalan equation over function fields was already analyzed in \cite{Silverman}, where the main theorem claims that the generalized Catalan equation has no solutions for $m$ and $n$ sufficiently large. Unfortunately, it is not hard to produce counterexamples to the main theorem given there. Following the notation in \cite{Silverman}, we choose $k = \mathbb{F}_p$, $K = k(u)$, $a = x = u$, $b = y = 1 - u$ and $m = n = p^t - 1$ for $t \in \mathbb{Z}_{\geq 0}$. Then we have
\[
ax^m + by^n = u \cdot u^{p^{t} - 1} + (1 - u) \cdot (1 - u)^{p^{t} - 1} = 1
\] 
due to Frobenius, illustrating the need of (\ref{result}). \\

\noindent \textbf{Proof of Theorem 1} \\
For our proof we will need a generalization of Mason's ABC-theorem for function fields in one variable to an arbitrary number of variables. Such a result is given in \cite{HW}.

\begin{proof}
Let $(x, y, m, n)$ be an arbitrary solution. Without loss of generality we may assume that $\{z_1, \ldots, z_k\}$ forms a transcendence basis of $K/\mathbb{F}_p$.

We write $t := r - k$ and rename $z_{k + 1}, \ldots, z_r$ as $y_1, \ldots, y_t$ respectively. Define
\[
A_0 := \mathbb{F}_p[z_1, \ldots, z_k], \quad K_0 := \mathbb{F}_p(z_1, \ldots, z_k).
\]
Then
\[
A = A_0[y_1, \ldots, y_t], \quad K = K_0(y_1, \ldots, y_t).
\]
Note that $A$ and $K$ are respectively the coordinate ring and the function field of an affine algebraic variety $V$ over $\mathbb{F}_p$. Then the algebraic closure of $\mathbb{F}_p$ in $K$ is a finite extension of $\mathbb{F}_p$, say $\mathbb{F}_q$ with $q = p^n$ for some $n \in \Z_{>0}$. Because $\mathbb{F}_q$ is perfect, it follows that $K$ is separably generated over $\mathbb{F}_q$, see Theorem 26.2 in \cite{Matsumura}.

Taking the projective closure of $V$ and normalizing gives a projective variety non-singular in codimension one, still with function field $K$. Let us call the variety thus obtained $W$. Our goal will be to introduce a height on $K$. For later purposes it will be useful to do this in a slightly more general setting. So let $X$ be a projective variety, non-singular in codimension one, defined over a perfect field $k$. We write $L$ for the function field of $X$ and we assume that $k$ is algebraically closed in $L$.

Fix a projective embedding of $X$ such that $X \subseteq \mathbb{P}^M_k$ for some positive integer $M$. Then a prime divisor $\mathfrak{p}$ of $X$ over $k$ is by definition an irreducible subvariety of codimension one. Recall that for a prime divisor $\mathfrak{p}$ the local ring $\mathcal{O}_\mathfrak{p}$ is a discrete valuation ring, since $X$ is non-singular in codimension one. Following \cite{L2} we will define heights on $X$. To do this, we start by defining a set of normalized discrete valuations
\[
M_L := \{\text{ord}_\mathfrak{p} : \mathfrak{p} \text{ prime divisor of } X\},
\]
where $\text{ord}_\mathfrak{p}$ is the normalized discrete valuation of $L$ corresponding to $\mathcal{O}_\mathfrak{p}$. If $v = \text{ord}_\mathfrak{p} \in M_L$, we define for convenience $\deg v := \deg \mathfrak{p}$ with $\deg \mathfrak{p}$ being the projective degree in $\mathbb{P}^M_k$. Then the set $M_L$ satisfies the sum formula for all $x \in L^\ast$
\[
\sum_v v(x) \deg v = 0.
\]
If $P$ is a point in $\mathbb{P}^r(L)$ with coordinates $(y_0 : \ldots : y_r)$ in $L$, then its (logarithmic) height is
\[
h_L(P) = -\sum_{v} \min_i \{v(y_i)\} \deg v.
\]
Furthermore we define for an element $x \in L$
\begin{align}
\label{height2}
h_L(x) = h_L(1 : x).
\end{align}
We will need the following properties of the height.

\begin{lemma}
\label{lheight}
Let $x, y \in L$ and $n \in \Z$. The height defined by (\ref{height2}) has the following properties:
\begin{enumerate}
\item[(a)] $h_L(x) = 0 \Leftrightarrow x \in k$;
\item[(b)] $h_L(x + y) \leq h_L(x) + h_L(y)$;
\item[(c)] $h_L(xy) \leq h_L(x) + h_L(y)$;
\item[(d)] $h_L(x^n) = nh_L(x)$;
\item[(e)] Suppose that $k$ is a finite field and let $C > 0$ be given. Then there are only finitely many $x \in L^\ast$ satisfying $h_L(x) \leq C$;
\item[(f)] $h_L(x) = h_{\overline{k} \cdot L}(x)$.
\end{enumerate}
\end{lemma}

\begin{proof}
Property (a) is Proposition 4 of \cite{L1} (p. 157), while properties (b), (c) and (d) are easily verified. Property (e) is proven in \cite{Masser}. Finally, property (f) can be found after Proposition 3.2 in \cite{L2} (p. 63).
\end{proof}

Let us now dispose with the case $ax^m \in \mathbb{F}_q$. Then
\[
2h_K(x) \leq mh_K(x) = h_K(x^m) \leq h_K(ax^m) + h_K(a^{-1}) = h_K(a^{-1}),
\]
hence there are only finitely many possibilities for $x$. Now observe that $ax^m \in \mathbb{F}_q$ implies $by^n \in \mathbb{F}_q$. By the same argument we get finitely many possibilities for $y$, so we are done in this case.

From now on we will assume $ax^m \not \in \mathbb{F}_q$ and hence $by^n \not \in \mathbb{F}_q$. Then it follows that
\[
h_K(ax^m), h_K(by^n) \neq 0.
\]
Write
\[
ax^m = \gamma^{p^t}, by^n = \delta^{p^s}
\]
for some $t, s \in \mathbb{Z}_{\geq 0}$ and $\gamma, \delta \not \in K^p$. After substitution we get
\[
\gamma^{p^t} + \delta^{p^s} = 1.
\]
Extracting $p$-th roots gives $t = s$ and hence
\begin{align}
\label{Mason}
\gamma + \delta = 1.
\end{align}
Our goal will be to apply the main theorem of \cite{HW} to (\ref{Mason}). For completeness we repeat it here.

\begin{theorem}
\label{HWT}
Let $X$ be a projective variety over an algebraically closed field $k$ of characteristic $p > 0$, which is non-singular in codimension one. Let $L = k(X)$ be its function field and let $M_L$ be as above. Let $L_1, \ldots, L_q$, $q \geq n + 1$, be linear forms in $n + 1$ variables over $k$ which are in general position. Let $\mathbf{X} = (x_0 : \ldots : x_n) \in \mathbb{P}^n(L)$ be such that $x_0, \ldots, x_n$ are linearly independent over $K^{p^m}$ for some $m \in \N$. Then, for any fixed finite subset $S$ of $M_L$, the following inequality holds:
\begin{align*}
&(q - n - 1) h(x_0 : \ldots : x_n) \\
& \ \leq \sum_{i = 1}^q \sum_{v \not \in S} \deg v \min\{np^{m - 1}, v(L_i(\mathbf{X})) - \min_{0 \leq j \leq n} \{v(x_j)\}\} \\
& \ \ \ + \frac{n(n + 1)}{2} p^{m - 1} \left(C_X  + \sum_{v \in S} \deg v \right), 
\end{align*}
where $C_X$ is a constant depending only on $X$.
\end{theorem}
\begin{proof}
See Theorem in \cite{HW}.
\end{proof}

Note that Theorem \ref{HWT} requires that the ground field $k$ is algebraically closed. But a constant field extension does not change the height by Lemma \ref{lheight}(f). Hence we can keep working with our field $K$ instead of $\overline{\mathbb{F}_p} \cdot K$. Define the following three linear forms in two variables $X, Y$
\begin{align*}
L_1 &= X \\
L_2 &= Y \\
L_3 &= X + Y.
\end{align*}
We apply Theorem \ref{HWT} with our $W$, the above $L_1, L_2, L_3$ and $\mathbf{X} = (\gamma : \delta) \in \mathbb{P}^1(K)$. We claim that $\gamma$ and $\delta$ are linearly independent over $K^p$. Suppose that there are $e, f \in K^p$ such that
\[
e \gamma + f \delta = 0.
\]
Together with $\gamma + \delta = 1$ we find that
\[
0 = e \gamma + f \delta = e (1 - \delta) + f \delta = e + (f - e) \delta.
\]
If $e \neq f$, then this would imply that $\delta \in K^p$, contrary to our assumptions. Hence $e = f$, but then we find
\[
0 = e \gamma + f \delta = e
\]
and we conclude that $e = f = 0$ as desired.

We still have to choose the subset $S$ of $M_K$ to which we apply Theorem \ref{HWT}. First we need to make some preparations. From now on $v$ will be used to denote an element of $M_K$. Define
\begin{align*}
N_0 &:= \{v : v(a) \neq 0 \vee v(b) \neq 0\} \\
N_1 &:= \{v : v(a) = 0, v(b) = 0, v(\gamma) > 0\} \\
N_2 &:= \{v : v(a) = 0, v(b) = 0, v(\delta) > 0\} \\
N_3 &:= \{v : v(a) = v(b) = 0, v(\gamma) = v(\delta) < 0\}.
\end{align*}

It is clear that $N_0$, $N_1$, $N_2$ and $N_3$ are finite disjoint sets. Before we proceed, we make a simple but important observation in the form of a lemma.

\begin{lemma}
\label{gd}
Let $(\gamma, \delta)$ be a solution of (\ref{Mason}). If $v(\gamma) < 0$ or $v(\delta) < 0$, then
\[
v(\gamma) = v(\delta) < 0.
\]
\end{lemma}

\begin{proof}
Obvious.
\end{proof}

Recall that
\[
h_K(\gamma) = \sum_v \max(0, v(\gamma)) \deg v = \sum_v -\min(0, v(\gamma)) \deg v
\]
and
\[
h_K(\delta) = \sum_v \max(0, v(\delta)) \deg v = \sum_v -\min(0, v(\delta)) \deg v.
\]
Lemma \ref{gd} tells us that
\[
\sum_v -\min(0, v(\gamma)) \deg v= \sum_v -\min(0, v(\delta)) \deg v,
\]
hence
\begin{align}
\label{height}
h_K(\gamma) = h_K(\delta) &= \sum_v \max(0, v(\gamma)) \deg v = \sum_v -\min(0, v(\gamma)) \deg v \\
&= \sum_v \max(0, v(\delta)) \deg v = \sum_v -\min(0, v(\delta)) \deg v.
\end{align}
We will use these different expressions for the height throughout. Let us now derive elegant upper bounds for $N_1$, $N_2$ and $N_3$. Again we will phrase it as a lemma.

\begin{lemma}
\label{N}
Let $(\gamma, \delta)$ be a solution of (\ref{Mason}). Then 
\begin{align*}
h_K(\gamma) &= h_K(\delta) \geq m \sum_{v \in N_1} \deg v, \\
h_K(\gamma) &= h_K(\delta) \geq n \sum_{v \in N_2} \deg v, \\
h_K(\gamma) &= h_K(\delta) \geq \text{lcm}(m, n) \sum_{v \in N_3} \deg v.
\end{align*}
\end{lemma}
\begin{proof}
We know that
\[
h_K(\gamma) = h_K(\delta) = \sum_v \max(0, v(\gamma)) \deg v \geq \sum_{v \in N_1} \max(0, v(\gamma)) \deg v.
\]
Now let $v \in N_1$. This means that $v(a) = v(b) = 0$ and $v(\gamma) > 0$. Then $ax^m = \gamma^{p^t}$ implies
\[
v(a) + mv(x) = p^t v(\gamma)
\]
and hence $mv(x) = p^t v(\gamma)$. But $m$ and $p$ are coprime by assumption, so we obtain $m \mid v(\gamma)$. Because $v(\gamma) > 0$, this gives $v(\gamma) \geq m$ and we conclude that
\[
h_K(\gamma) = h_K(\delta) \geq m \sum_{v \in N_1} \deg v.
\]
Using
\[
h_K(\gamma) = h_K(\delta) = \sum_v \max(0, v(\delta)) \deg v \geq \sum_{v \in N_2} \max(0, v(\delta)) \deg v,
\]
we find in a similar way that
\[
h_K(\gamma) = h_K(\delta) \geq n \sum_{v \in N_2} \deg v.
\]
It remains to be proven that
\[
h_K(\gamma) = h_K(\delta) \geq \text{lcm}(m, n) \sum_{v \in N_3} \deg v.
\]
Now we use
\begin{align*}
h_K(\gamma) = h_K(\delta) &= \sum_v - \min(0, v(\gamma)) \deg v = \sum_v - \min(0, v(\delta)) \deg v \\
&\geq \sum_{v \in N_3} - \min(0, v(\gamma)) \deg v = \sum_{v \in N_3} - \min(0, v(\delta)) \deg v.
\end{align*}
Now take $v \in N_3$. Then $v(\gamma) = v(\delta) < 0$. In the same way as before, we can show that $m \mid v(\gamma)$ and $n \mid v(\delta)$. But $v(\gamma) = v(\delta) < 0$ by Lemma \ref{gd}, so we find that
\[
h_K(\gamma) = h_K(\delta) \geq \text{lcm}(m, n) \sum_{v \in N_3} \deg v
\]
as desired.
\end{proof}

Define
\[
S := N_0 \cup N_1 \cup N_2 \cup N_3.
\]
Suppose that $v \not \in S$. We claim that
\[
v(\gamma) = v(\delta) = 0.
\]
But $v \not \in S$ implies $v \not \in N_0$, so certainly $v(a) = v(b) = 0$. Furthermore, we have that $v \not \in N_1$ and $v \not \in N_2$, which means that $v(\gamma) \leq 0$ and $v(\delta) \leq 0$. If $v(\gamma) < 0$ or $v(\delta) < 0$, then Lemma \ref{gd} gives $v \in N_3$, contradicting our assumption $v \not \in S$. Hence $v(\gamma) = v(\delta) = 0$ as desired.

From our claim it follows that we have for $v \not \in S$ and $i = 1, 2, 3$
\[
v(L_i(\gamma, \delta)) = \min(v(\gamma), v(\delta)).
\]
Theorem \ref{HWT} tells us that
\[
h_K(\gamma : \delta) \leq C_W + \sum_{v \in S} \deg v,
\]
where $C_W$ is a constant depending on $W$ only. By Lemma \ref{N} we find that
\begin{align*}
\sum_{v \in S} \deg v &= \sum_{v \in N_0} \deg v + \sum_{v \in N_1} \deg v + \sum_{v \in N_2} \deg v + \sum_{v \in N_3} \deg v \\
&\leq C_{a, b} + \left(\frac{1}{m} + \frac{1}{n} + \frac{1}{\text{lcm}(m, n)}\right) h_K(\gamma),
\end{align*}
where $C_{a, b}$ is a constant depending on $a$ and $b$ only. Now (\ref{mn}) implies
\[
\frac{1}{m} + \frac{1}{n} + \frac{1}{\text{lcm}(m, n)} < 0.9,
\]
hence
\[
h_K(\gamma : \delta) \leq 10(C_W + C_{a, b}).
\]
But $\gamma + \delta = 1$ gives
\[
h_K(\gamma) = h_K(\delta) = h_K(\gamma : \delta).
\]
The theorem now follows from Lemma \ref{lheight}(e).
\end{proof}

\newpage

\noindent \textbf{Discussion of Theorem 1} \\
The conclusion of Theorem 1 tells us that there is a finite set $\mathcal{T} \subseteq K^2$ such that for any solution $(x, y, m, n)$ of (\ref{Cat}), there is a $(\gamma, \delta) \in \mathcal{T}$ and $t \in \mathbb{Z}_{\geq 0}$ such that
\[
ax^m = \gamma^{p^t}, by^n = \delta^{p^{t}}.
\]
Since $\mathcal{T}$ is finite, we may assume that $\gamma$ and $\delta$ are fixed in the above two equations. The resulting equation is studied in \cite{PowerEquation}.


\begin{thebibliography}{9}
\bibitem{Brindza} 
B. Brindza,
\textit{The Catalan equation over finitely generated integral domains,}
Publ. Math. Debrecen 42 (1993),
193-198.

\bibitem{HW}
L.-C. Hsia, J.T.-Y. Wang,
\textit{The ABC theorem for higher-dimensional function fields,}
Trans. Amer. Math. Soc. 356 (2004),
2871-2887.

\bibitem{L1}
S. Lang,
\textit{Introduction to Algebraic Geometry,}
Addison-Wesley, Redwood City, CA, 1973.

\bibitem{L2}
S. Lang,
\textit{Fundamentals of Diophantine Geometry,}
Springer, Berlin, 1983.

\bibitem{Masser}
D.W. Masser,
\textit{Mixing and Linear Equations over Groups in positive characteristic,}
Isreal Journal of Mathematics 142 (2004),
189-204.

\bibitem{MThesis}
P. Koymans,
\textit{The Catalan equation,}
\url{http://pub.math.leidenuniv.nl/~koymansph/MasterThesis.pdf}

\bibitem{PowerEquation}
P. Koymans,
\textit{Power Equation,}
\url{http://pub.math.leidenuniv.nl/~koymansph/PowerEquation.pdf}

\bibitem{Silverman} 
J.H. Silverman,
\textit{The Catalan equation over function fields,}
Trans. Amer. Math. Soc. 273 (1982),
201-205.

\bibitem{Matsumura}
H. Matsumura,
\textit{Commutative ring theory,}
Cambridge University Press, Cambridge, UK, 1986.
\end{thebibliography}
\end{document}